\documentclass[envcountsame]{llncs}
\usepackage{amsmath}
\usepackage{amssymb}
\usepackage{enumitem}

\def\squareforqed{\hbox{\rlap{$\sqcap$}$\sqcup$}}
\def\qed{\ifmmode\squareforqed\else{\unskip\nobreak\hfil
\penalty50\hskip1em\null\nobreak\hfil\squareforqed
\parfillskip=0pt\finalhyphendemerits=0\endgraf}\fi}

\def\lsigma#1{{\Sigma}^0_{#1}}
\def\lpi#1{{\Pi}^0_{#1}}

\def\QCB{{{\bf QCB}_0}}
\def\IN{{\mathbb N}}

\def\calC{{\mathcal{C}}}
\def\calD{{\mathcal{D}}}
\def\calE{{\mathcal{E}}}
\def\calF{{\mathcal{F}}}
\def\calG{{\mathcal{G}}}

\def\calP{{\mathcal{P}}}

\def\I#1{{{\mathrm{\mathbf I}}(#1)}}

\newcommand{\ACA}{\mathsf{ACA_0}}

\newcommand{\cntSets}{{\mathsf{ODS}}}

\def\QCB{{{\bf QCB}_0}}

\newcommand{\Mor}{\mathsf{Mor}}
\newcommand{\Obj}{\mathsf{Obj}}

\def\src{{\sigma}}
\def\tar{{\tau}}
\def\id{{\iota}}

\def\act{{\mathrm{\bold{Act}}}}

\def\sierp{{\mathbb S}}

\def\int{{\mathrm{\mathbf{int}}}}

\def\wayabovearrow{\rlap{\raise-.25ex\hbox{$\shortuparrow$}}\raise.25ex\hbox{$\shortuparrow$}}
\def\waybelowarrow{\rlap{\raise.25ex\hbox{$\shortdownarrow$}}\raise-.25ex\hbox{$\shortdownarrow$}}

\begin{document}

\title{A note on computable \'{e}tale spaces}
\author{Matthew de Brecht\thanks{The author thanks Ruiyuan Chen for helpful discussions. This note contains proofs of the main results of the author's presentation at the 2025 RIMS workshop on proof theory and computability theory, which was organized by T.~Kihara. This work was supported by JSPS Bilateral Program Number JPJSBP120265001.}}
\institute{Graduate School of Human and Environmental Studies, Kyoto University, Japan\\
\email{matthew@i.h.kyoto-u.ac.jp}}

\maketitle

\begin{abstract}
An \'{e}tale space over a topological space $Y$ is defined as a local homeomorphism from a topological space $X$ into $Y$. They often come up in topos theory because of the equivalence between sheaves and \'{e}tale spaces over a space. In this note, we define computable \'{e}tale spaces over a computable topological space $Y$ within the TTE framework of computable topology, and show they are naturally equivalent to computable functions from $Y$ to $\cntSets$, the effective quasi-Polish category of overt-discrete quasi-Polish spaces. More generally, if $\calC$ is a computable category (or groupoid), then there is an equivalence between computable functors from $\calC$ to $\cntSets$, and computable \'{e}tale spaces equipped with a computable action by $\calC$. 
\end{abstract}


\section{Introduction}

We recommend \cite{V07} for an intuitive explanation of the connections between Grothendiek toposes and geometric logic. Briefly, frames (complete lattices satisfying a certain infinitary distributive law) correspond to the Lindenbaum-algebras of propositional geometric theories, and Grothendiek toposes in a sense correspond to the ``Lindenbaum algebras'' of predicate geometric theories. A famous result by Joyal and Tierney \cite{JT84} shows that every Grothendiek topos is equivalent to the category of continuous actions of a localic groupoid.

An analogous connection between $\omega_1$-pretoposes and $\omega_1$-coherent theories was established by R.~Chen in \cite{Chen19,Chen25}. An $\omega_1$-pretopos is a generalization of a Grothendiek topos where only countable coproducts are required to exist, rather than all small coproducts. Similarly, $\omega_1$-coherent logic is the restriction of geometric logic that requires infinitary disjunctions to be countable. R.~Chen showed that the syntactic $\omega_1$-pretopos (the ``Lindenbaum algebra'') of a countable $\omega_1$-coherent theory is equivalent to the category $\act(\calC)$ of (countably based) continuous actions of a quasi-Polish groupoid $\calC$. This can be viewed as a  generalization (from propositional logic to predicate logic) of R.~Heckmann's spatiality result \cite{H15}, which implies the Lindenbaum algebra of a countable $\omega_1$-coherent propositional theory is isomorphic to the frame of opens of a quasi-Polish space.

This note is a contribution towards effectivizing R.~Chen's result and other aspects of Grothendiek topos theory. We define computable \'{e}tale spaces over a computable topological space $Y$ within the TTE framework of computable topology \cite{W00}, and show they are equivalent to computable functions from $Y$ to $\cntSets$, the effective quasi-Polish category of overt discrete quasi-Polish spaces. More generally, we show that for any computable category $\calC$ there is a computability preserving equivalence between the category $\act(\calC)$ of (countably based) \'{e}tale spaces with continuous actions by $\calC$, and the functor category $\cntSets^\calC$ of (topologically) continuous functors from $\calC$ to $\cntSets$. It follows that the syntactic $\omega_1$-pretopos of a countable $\omega_1$-coherent theory is equivalent to a functor category $\cntSets^\calC$ for some quasi-Polish category $\calC$. These syntactic $\omega_1$-pretoposes are analogous to classifying (Grothendiek) toposes of geometric theories, and the categorically equivalent functor categories $\cntSets^\calC$ are suitable for computability theory because they are internal categories of $\QCB$.

See \cite{HMMM17,MPSS18,dbr21b,FNPPV25} for related work applying computability theory to category theoretical structures. See also \cite{K23} for connections between Lawvere-Tierney topologies and computability theory. We mainly work with computable topological spaces\footnote{Readers uninterested in computability can interpret ``computable topological space'' as ``countably based space'', and ``computable function'' as ``continuous function''.}, but will need $\QCB$-spaces \cite{Sch21} for functor categories. One could also generalize further to represented spaces \cite{P16}, which appears necessary to internalize some categories of interest \cite{dbr21b}. We will also emphasize the constructions for quasi-Polish spaces, because they are exactly what is needed to establish the connections with $\omega_1$-coherent logic. By encoding quasi-Polish spaces as spaces of ideals and using the techniques of \cite{mummert:phd}, most of our results can be formalized within the subsytem $\ACA$ of second order arithmetic (or even weaker systems).

\section{Preliminaries}

\subsection{Computable topological spaces}

Let $\prec$ be a transitive relation on $\IN$. A subset $I\subseteq \IN$ is an \emph{ideal} if and only if:
\begin{enumerate}
\item
$I \not=\emptyset$,\hfill (\emph{$I$ is non-empty})
\item
$(\forall a \in I) (\forall b \in \IN)\, (b \prec a \Rightarrow b \in I)$,\hfill (\emph{$I$ is a lower set})
\item
$(\forall a,b \in I)(\exists c\in I)\, (a \prec c  \,\&\,  b \prec c)$.\hfill (\emph{$I$ is directed})
\end{enumerate}
The collection $\I{\prec}$ of all ideals has the topology generated by basic open sets of the form $[a]_{\prec} = \{ I \in \I{\prec} \mid a \in I\}$.

A topological space is \emph{quasi-Polish} if it is homeomorphic to a space of the form $\I{\prec}$ for some transitive relation $\prec$ on $\IN$. A represented space is \emph{effectively quasi-Polish} if it is computably isomorphic to a space of the form $\I{\prec}$ for some computably enumerable (c.e.) transitive relation $\prec$ on $\IN$. An effectively quasi-Polish space is \emph{computably overt} if furthermore the corresponding relation $\prec$ is such that $E_\prec=\{ a\in\IN \mid [a]_\prec\not=\emptyset\}$ is c.e. A computably overt effective quasi-Polish space is called a \emph{computable quasi-Polish space}.
 
Quasi-Polish spaces were introduced in \cite{dbr}, and the equivalent characterization in terms of spaces of ideals that we use here was first shown in \cite{DPS}. Effective quasi-Polish spaces have been studied in \cite{Sel15,KK17,DPS,HRSS19,DKS24}.

An admissible representation \cite{Sch02} of $\I{\prec}$ represents each ideal by an enumeration of its elements. A function $f\colon \I{\prec_1} \to \I{\prec_2}$ is \emph{computable} iff there is a Type Two Turing Machine \cite{W00} which when given any enumeration of the elements of any $I\in \I{\prec_1}$ outputs an enumeration of $f(I) \in \I{\prec_2}$.

A \emph{c.e. open subset} (or \emph{$\lsigma 1$-subset}) of $\I{\prec}$ is an open subset of the form $\bigcup_{a\in S} [a]_{\prec}$ for some c.e. subset $S\subseteq \IN$. A $\lpi 2$-subset is a subset of the form
\[ \{ I \in \I{\prec} \mid (\forall i\in\IN)[ I \in U_i \Rightarrow I \in V_i]\}\]
for some c.e. sequence $(U_i, V_i)_{i\in\IN}$ of c.e. open subsets of $\I{\prec}$. Effective quasi-Polish spaces are closed under $\lpi 2$-subspaces (this is a well-known result, but see the proof of Theorem~3 in \cite{dbr20} for how to prove this in $\ACA$).

A pair $(\prec, X)$, where $\prec$ is a c.e. transitive relation and $X\subseteq \I{\prec}$,  is called a \emph{computable topological space}\footnote{See Section 4.2 of \cite{dbr20} for a proof that this definition is computably equivalent to the more common definition of a \emph{computable (or effective) topological space}.}. Computable functions, c.e. open subsets, and $\lpi 2$-subsets of computable topological spaces are defined by restricting the definition for effective quasi-Polish spaces to their subspaces. Computable topological spaces are countably based, and conversely every countably based space is homeomorphic to a computable topological space for some suitable $\prec$.

\subsection{Computable categories}

A \emph{computable category} is an internal category \cite{J99} in the category of computable topological spaces and computable functions. More concretely, a computable category is a tuple $\calC = (\calC_\Obj, \calC_\Mor, \src,\tar,\id,\circ)$ consisting of the following data:
\begin{itemize}
\item
$\calC_\Obj$ (objects) and $\calC_\Mor$ (morphisms) are computable topological spaces.
\item
$\src\colon \calC_\Mor\to\calC_\Obj$ (source), $\tar\colon \calC_\Mor\to\calC_\Obj$ (target), and $\id\colon \calC_\Obj\to\calC_\Mor$ (identity) are computable functions.
\item
$\circ:\subseteq \calC_\Mor\times\calC_\Mor\to\calC_\Mor$ (composition) is a partial computable function with domain
\[dom(\circ) = \{ \langle g,f \rangle \in \calC_\Mor\times\calC_\Mor \mid \src(g) = \tar(f)\}\]
\end{itemize}
subject to the following:
\begin{itemize}
\item
$\src(g\circ f) = \src(f)$ and $\tar(g\circ f) = \tar(g)$,
\item
$\src(\id(c)) = c$ and $\tar(\id(c)) = c$,
\item
$(h\circ g)\circ f = h \circ (g\circ f)$ when the compositions $h\circ g$ and $g\circ f$ are defined,
\item
if $\src(f) = c$ and $\tar(f) = d$ then $f\circ \id(c) = f = \id(d)\circ f $.
\end{itemize}
We write $f\colon c\to d$ to denote that $f$ is a morphism with  $\src(f) = c$ and $\tar(f) = d$.

We call $\calC$ an \emph{effective quasi-Polish category} if $\calC_\Obj$ and $\calC_\Mor$ are effective quasi-Polish spaces. This implies $\calC$ is an internal category in the category of effective quasi-Polish spaces and computable functions, because effective quasi-Polish spaces are closed under finite limits (in particular, the pullback $dom(\circ)$ is an effective quasi-Polish space). $\calC$ is \emph{locally compact} if $\calC_\Mor$ and $\calC_\Obj$ are locally compact spaces.

As a useful example, note that every computable topological space $X$ can be viewed as a computable category $\calC_X$ by defining $(\calC_X)_\Mor = (\calC_X)_\Obj = X$, and defining $\src$, $\tar$, and $\id$ to be the identity functions. The only morphisms in $\calC_X$ are the identity morphisms, so composition $\circ:\subseteq X\times X \to X$ is the (partial) inverse of the diagonal embedding $\Delta\colon X\to X\times X$, $x\mapsto \langle x,x\rangle$. Then $\calC_X$ is a discrete category, although not necessarily discrete in the topological sense.

\subsection{Computable functors}

Let $\calC= (\calC_\Obj, \calC_\Mor, \src_\calC,\tar_\calC,\id_\calC,\circ_\calC)$ and $\calD= (\calD_\Obj, \calD_\Mor, \src_\calD,\tar_\calD,\id_\calD,\circ_\calD)$ be computable categories. A \emph{computable functor} from $\calC$ to $\calD$ is a pair $F = (F_{\Obj}, F_{\Mor})$ of computable functions $F_{\Obj} \colon \calC_\Obj\to\calD_\Obj$ and $F_{\Mor} \colon \calC_\Mor\to\calD_\Mor$ satisfying
\begin{enumerate}
\item
$F_{\Obj}\circ \src_\calC = \src_\calD \circ F_{\Mor}$,
\item
$F_{\Obj}\circ \tar_\calC = \tar_\calD \circ F_{\Mor}$,
\item
$F_{\Mor}\circ \id_\calC = \id_\calD \circ F_{\Obj}$, and
\item
$F_{\Mor}(g\circ_\calC f) = F_{\Mor}(g) \circ_\calD F_{\Mor}(f)$ for all composable $f,g\in\calC_\Mor$.
\end{enumerate}
A computable functor $F$ from $\calC$ to $\calD$ will be denoted $F\colon \calC\to\calD$. It is clear that the identity functor $1_\calC\colon\calC\to\calC$ (which is the identity on objects and morphisms) is computable, and that computable functors are closed under composition.

Given computable functors $F,G\colon \calC\to\calD$, a \emph{computable natural transformation} from $F$ to $G$ is given by a computable function $\eta\colon \calC_\Obj \to \calD_\Mor$ satisfying:
\begin{itemize}
\item
$\src_{\calD}\circ\eta = F_{\Obj}$ and $\tar_{\calD}\circ \eta = G_{\Obj}$ 

(i.e., $\eta(c) \colon F_{\Obj}(c) \to G_{\Obj}(c)$ for each $c\in\calC_\Obj$), 
\item
 $\eta(\tar_\calC(f)) \circ_\calD F_{\Mor}(f) = G_{\Mor}(f) \circ_\calD \eta(\src_\calC(f))$ for each $f\in\calC_\Mor$.
\end{itemize}
Usually we will write $\eta\colon F\to G$ if it will not cause any confusion.

A \emph{continuous functor}\footnote{We use \emph{continuous} in the topological sense, which conflicts with the established usage of \emph{continuous functor} as a functor that preserves all small limits. This will not cause a problem because we do not use the latter notion in this note.} is a functor that is computable with respect to an oracle. Similarly, a \emph{continuous natural transformation} is a natural transformation that is computable with respect to an oracle. We will need these relativized versions  when defining functor categories in the next section. 

\subsection{Functor categories}

Since the category of computable topological spaces (or quasi-Polish spaces) is not cartesian closed, when we discuss functor categories we will potentially need to step into the larger cartesian closed category $\QCB$ \cite{Sch21}. 

The functor category $\calD^\calC$ has continuous functors from $\calC$ to $\calD$ as objects and continuous natural transformations as morphisms. It is formally defined as:
\begin{itemize}
\item
$(\calD^\calC)_\Obj$ is the subspace of $(\calD_\Obj)^{\calC_\Obj}\times (\calD_\Mor)^{\calC_\Mor}$ of continuous functors.
\item
$(\calD^\calC)_\Mor$ is the subspace of $(\calD_\Mor)^{\calC_\Obj} \times (\calD^\calC)_\Obj\times (\calD^\calC)_\Obj$ of tuples $\langle \eta, F,G\rangle$ where $\eta$ is a continuous natural transformation from $F$ to $G$.
\item
$\src_{\calD^\calC}(\langle \eta, F,G\rangle) = F$, $\tar_{\calD^\calC}(\langle \eta, F,G\rangle)=G$, $\id_{\calD^\calC}(F) =\langle \id_\calD\circ F_{\Obj}, F,F\rangle$.
\item
$\langle \theta, G,H\rangle \circ_{\calD^\calC} \langle \eta, F,G\rangle = \langle \zeta, F,H\rangle$, where $\zeta(c) = \theta(c) \circ_{\calD} \eta(c)$.
\end{itemize} 

The function spaces $(\calD_\Obj)^{\calC_\Obj}$ and $(\calD_\Mor)^{\calC_\Mor}$ are constructed as in $\QCB$, and the topology is the sequentialization of the compact-open topology. 

If $X$ and $Y$ are quasi-Polish, and $X$ is locally compact, then $Y^X$ is quasi-Polish and has the compact-open topology (see Theorem~16.3 of \cite{DGJL}). Therefore, if $\calC$ and $\calD$ are quasi-Polish categories and $\calC$ is locally compact, then $(\calD_\Obj)^{\calC_\Obj}$ and $(\calD_\Mor)^{\calC_\Mor}$ are quasi-Polish. Thus $\calD^\calC$ is a quasi-Polish category, because $(\calD^\calC)_\Obj$ and $(\calD^\calC)_\Mor$ can be constructed as equalizers of continuous functions between quasi-Polish spaces, and $\src_{\calD^\calC}$, $\tar_{\calD^\calC}$, and $\id_{\calD^\calC}$ are continuous.

\subsection{Computable \'{e}tale spaces}

An \emph{\'{e}tale space} over a topological space $Y$ is a local homeomorphism $p\colon X\to Y$, which means that each $x\in X$ has an open neighborhood $U$ such that the restriction $p|_U$ is a homeomorphism with an open subset of $Y$. A morphism between \'{e}tale spaces $p\colon X\to Y$ and $q\colon Z\to Y$ is a continuous function $f\colon X\to Z$ such that $q\circ f = p$. It is well-known that the category of \'{e}tale spaces over $Y$ is equivalent to the category of sheaves of sets over $Y$ (see \cite{MM92}).

We only consider countably based \'{e}tale spaces in this note. Having a countable basis forces the fibers to be countable, hence these can be viewed as sheaves of countable sets over a space. Note that a countably based \'{e}tale space over a quasi-Polish space is automatically quasi-Polish (Lemma~2.2 of \cite{Chen19}).

Let $X$ and $Y$ be computable topological spaces. A \emph{computable \'{e}tale space} consists of a computable function $p\colon X\to Y$ along with the following:
\begin{enumerate}
\item
a computable enumeration of c.e. open subsets $(U_n)_{n\in\IN}$ that cover $X$,
\item
a computable enumeration of c.e. open subsets $(V_n)_{n\in\IN}$ of $Y$,
\item
a computable enumeration of computable functions (sections) $s_n\colon V_n \to U_n$ ($n\in\IN$) satisfying 
\begin{itemize}
\item
$s_n(p(x)) = x$ for each $x\in U_n$, and
\item
$p(s_n(y)) = y$ for each $y\in V_n$.
\end{itemize}
(i.e., each $s_n$ is a computable inverse of the restriction $p|_{U_n}\colon U_n\to V_n$). 
\end{enumerate}
In this case, we call $p\colon X\to Y$ a \emph{computable local homeomorphism}. We will sometimes call $X$ a computable \'{e}tale space if omitting the computable local homeomorphism $p$ and the computable sections does not cause confusion.

We have the following easy fact that we write here as a lemma for convenience.

\begin{lemma}\label{lem:sections}
If $p\colon X\to Y$ is a local homeomorphism with sections $s_n\colon V_n \to U_n$  ($n\in\IN$)  and $p(x)=y$, then $x=s_n(y)$ if and only if $x\in U_n$.
\qed
\end{lemma}

\subsection{Computable actions}

The definitions in this section are a minor generalization of the continuous groupoid actions in \cite{Chen19} to continuous actions by topological categories. 

Let $\calC$ be a computable category, and $p\colon X \to \calC_\Obj$ be a computable function from a computable topological space $X$ to the space of objects of $\calC$. A \emph{computable action} by $\calC$ on $X$ is a partial computable function $\alpha:\subseteq \calC_\Mor \times X \to X$ with domain
\[dom(\alpha) = \{\langle f,x\rangle \in \calC_\Mor \times X \mid \src(f) = p(x)\}\]
that satisfies the following conditions (where we abbreviate $\alpha(f,x)$ by $f\cdot x$):
\begin{enumerate}
\item
$\tar(f) = p(f\cdot x)$
\item
$\id(p(x)) \cdot x = x$
\item
$(g\circ f)\cdot x = g\cdot (f\cdot x)$.
\end{enumerate}

Intuitively, $\alpha$ interprets (in a computable way) each morphism $f\colon c \to d$ in $\calC$ as a function $\alpha(f,-) \colon p^{-1}(c) \to p^{-1}(d)$ between the fibers of $c$ and $d$, while preserving identity morphisms and composition.

A \emph{computable equivariant function} from $p\colon X \to \calC_\Obj$ to $q\colon Y \to \calC_\Obj$ (both equipped with computable actions by $\calC$) is a computable function $f\colon X\to Y$ such that $p = q\circ f$ and $f(g\cdot x) = g\cdot f(x)$.

If $\calC$ is a computable category, then a computable \'{e}tale space over $\calC_\Obj$ equipped with a computable action by $\calC$ is called a \emph{computable $\calC$-set}.

These definitions relativize to countably based \'{e}tale spaces, continuous actions, and continuous equivariant functions. We write $\act(\calC)$ for the category of countably based $\calC$-sets and continuous equivariant functions.

\subsection{The category of overt discrete quasi-Polish spaces}

Let $\calP(\IN)$ be the powerset of the natural numbers with the Scott-topology. If $(F_i)_{i\in\IN}$ is a computable enumeration of all finite subsets of $\IN$ and we define $i\prec j \iff F_i\subseteq F_j$, then $\calP(\IN)$ is computably homeomorphic to $\I{\prec}$. The space $\calP(\IN\times\IN)$ of sets of pairs of natural numbers is defined similarly, and is computably homeomorphic to $\calP(\IN)$.

A symmetric transitive relation is called a \emph{partial equivalence relation} \cite{J99}. The effective quasi-Polish category $\cntSets$ of overt discrete quasi-Polish spaces is defined as follows (compare with Section~4 of \cite{Chen19}):

\begin{itemize}
\item
$\cntSets_\Obj$ is the subspace of $\calP(\IN\times\IN)$ of all partial equivalence relations.
\item
$\cntSets_\Mor$ is the subspace of $\calP(\IN\times \IN) \times\cntSets_\Obj\times\cntSets_\Obj$ of tuples $\langle G, \equiv_\src, \equiv_\tar\rangle$  satisfying (for all $a,a',b,b'\in\IN$):
\begin{enumerate}
\item
$G(a,b)$ implies $a\equiv_\src a \,\&\, b\equiv_\tar b$.
\item
$G(a,b) \,\&\, a\equiv_\src a'$ implies $G(a',b)$.
\item
$G(a,b) \,\&\, b\equiv_\tar b'$ implies $G(a,b')$.
\item
$G(a,b) \,\&\, G(a,b')$ implies $b\equiv_\tar b'$.
\item
$a\equiv_\src a$ implies $(\exists b\in\IN)\, G(a,b)$.
\end{enumerate}
\item
$\src(\langle G, \equiv_\src, \equiv_\tar\rangle) = \equiv_\src$.
\item
$\tar(\langle G, \equiv_\src, \equiv_\tar\rangle) = \equiv_\tar$.
\item
$\id(\equiv)=\langle \equiv, \equiv, \equiv\rangle$.
\item
Composition $\circ$ is defined as 
\[\langle G', \equiv_\rho, \equiv_\tar\rangle \circ \langle G, \equiv_\src, \equiv_\rho\rangle = \langle \{\langle a,c\rangle\mid (\exists b\in\IN)\, G(a,b) \,\&\, G'(b,c)\} , \equiv_\src, \equiv_\tar\rangle,\]
which is easily seen to be well-defined and computable.
\end{itemize}

Clearly, $\cntSets_\Obj$ and $\cntSets_\Mor$ are effective quasi-Polish spaces because they are defined as $\lpi 2$-subspaces of effective quasi-Polish spaces.

An effective quasi-Polish space $X\cong \I{\prec}$ is \emph{computably discrete} if there is a c.e. set $D_\prec \subseteq \IN\times\IN$ such that $I = J$ in $\I{\prec}$ if and only if there exists $\langle a,b\rangle\in D_\prec$ with $a \in I$ and $b \in J$.

Let $\sierp=\{\bot,\top\}$ be the Sierpinski space. Then an effective quasi-Polish space $X$ is computably discrete if and only if equality $(=)\colon X\times X\to\sierp$, defined as ${=}(x,y)=\top \iff x=y$, is computable. Furthermore, $X$ is computably overt if and only if $(\exists_X)\colon \sierp^X\to\sierp$, defined as $\exists_X(U) =\top \iff U\not=\emptyset$, is computable. Therefore, our definitions are equivalent to the ones in \cite{ME04,P16}.

The next theorem is closely related to Theorem~1 of \cite{NPPV25}, which includes several other important results on overt discrete quasi-Polish spaces within the general context of represented spaces.

\begin{theorem}
$\cntSets$ is equivalent to the category of overt discrete quasi-Polish spaces. Under this equivalence, the computable points of $\cntSets_\Obj$ correspond to the computably overt computably discrete effective quasi-Polish spaces, and the computable points of $\cntSets_\Mor$ correspond to the computable functions between computably overt computably discrete effective quasi-Polish spaces.
\end{theorem}
\begin{proof}
We construct a full and faithful and essentially surjective functor $\calE$ from $\cntSets$ to the category of overt discrete quasi-Polish spaces. 
Given $\equiv$ in $\cntSets_\Obj$, define $\calE_\Obj(\equiv)=\I{\equiv}$, the space of ideals of $\equiv$.  This is quasi-Polish by definition, and discrete because if $I,J\in \calE_\Obj(\equiv)$ then $I = J$ if and only if $(\exists a\in \IN)[a\in I \,\&\, a\in J]$. Finally, $\calE_\Obj(\equiv)$ is overt (relative to $\equiv$) because $[a]_\prec\not = \emptyset$ if and only if $a\equiv a$. Clearly, if $\equiv$ is a computable point then $\calE_\Obj(\equiv)$ is a computably overt computably discrete effective quasi-Polish space.

$\calE_\Mor$ maps $\langle G, \equiv_\src, \equiv_\tar \rangle$ in $\cntSets_\Mor$ to the function $f_G\colon \calE_\Obj(\equiv_\src)\to \calE_\Obj(\equiv_\tar)$ defined as
\[f_G(I) = \{b\in \IN \mid (\exists a\in I)\, G(a,b)\}\]
for each $I\in \calE_\Obj(\equiv_\src)$. It follows from the definition of ideals and the five axioms defining elements of  $\cntSets_\Mor$ that $f_G(I)$ is an ideal of $\equiv_\tar$, hence $f_G$ is a well-defined function. Clearly, $f_G$ is computable from $\langle G, \equiv_\src, \equiv_\tar \rangle$. Note that a computable point of $\cntSets_\Mor$ has computable source and target because $\src$ and $\tar$ are computable.

It is easy to see that $\calE$ is a full and faithful functor. Next, we show that each computably overt computably discrete effective quasi-Polish space is computably homeomorphic to $\calE_\Obj(\equiv)$ for some computable point $\equiv$ of $\cntSets_\Obj$. 

Assume $\prec$ is a c.e. transitive relation, that $E_\prec=\{ a\in\IN \mid [a]_\prec\not=\emptyset\}$ is c.e., and that $D_\prec \subseteq \IN\times\IN$ is c.e. and such that $I = J$ in $\I{\prec}$ if and only if there is $\langle a,b\rangle\in D_\prec$ with $a \in I$ and $b \in J$. Without loss of generality, we can assume that $\langle a,b\rangle\in D_\prec$ implies $a\in E_\prec$ and $b\in E_\prec$.

Define
\[S = \{ c\in E_\prec \mid (\exists \langle a,b\rangle\in D_\prec)[a\prec c\,\&\, b\prec c]\}.\]  
Note that for each $c\in S$ there is a unique $I\in\I{\prec}$ such that $\{I\}=[c]_\prec$, and conversely for each $I\in\I{\prec}$ there is at least one $c\in S$ with $\{I\}=[c]_\prec$. If $[a]_\prec = [b]_\prec = \{I\}$ then $a$ and $b$ must have a $\prec$-upper bound in $I$, so by defining
\[a \equiv b \iff a,b\in S \,\&\, (\exists c\in S)[a\prec c\,\&\, b\prec c],\]
we obtain that $a\equiv b$ if and only if $a,b\in S$ and $[a]_\prec = [b]_\prec=\{I\}$ for a unique $I\in\I{\prec}$. Now define $g\colon \I{\prec}\to \calE_\Obj(\equiv)$ and $h\colon \calE_\Obj(\equiv)\to\I{\prec}$ as
\begin{eqnarray*}
g(I) &=& \{ a\in I \mid a \equiv a\}\\
h(J) &=& \{ b\in S\mid (\exists a\in J)\,b\prec a \}.
\end{eqnarray*}
Then $g(I)$ is the $\equiv$-ideal composed of all $a\in S$ with $[a]_\prec = \{I\}$. Furthermore, $h(J)$ is a $\prec$-ideal because each $\equiv$-ideal is $\prec$-directed. Therefore, $g$ and $h$ are well-defined, and easily seen to be computable inverses of each other. Therefore, $\calE_\Obj(\equiv)$ is computably homeomorphic to $\I{\prec}$. Since the proof relativizes, it follows that $\calE$ is essentially surjective.
\qed
\end{proof}

\section{Main result}

Our main result is the following:

\begin{theorem}
Let $\calC$ be a computable category. The category $\act(\calC)$ of countably based $\calC$-sets and continuous equivariant functions and the category $\cntSets^\calC$ of (topologically) continuous functors from $\calC$ to $\cntSets$ are equivalent. The equivalence preserves the computability of both objects and morphisms.
\qed
\end{theorem}

As a special case, if $X$ is a computable topological space then there is an equivalence between computable \'{e}tale spaces over $X$ and computable functions from $X$ to $\cntSets_\Obj$.  Thus, the functor category $\cntSets^{\calC_X}$ is equivalent to the category of countably based \'{e}tale spaces over $X$. 

We prove our main result in the next four sections by constructing functors $\calF\colon \cntSets^\calC \to \act(\calC)$ and $\calG\colon \act(\calC) \to \cntSets^\calC$ that preserve computable objects and morphisms, and then show that that $\calG\circ\calF$ and $\calF\circ\calG$ are naturally isomorphic to the identity functors.

\subsection{From $\cntSets^\calC$ to $\act(\calC)$}\label{subsec:functor2espace}

The next two subsections define the functor $\calF\colon \cntSets^\calC \to \act(\calC)$.

\subsection*{Objects (functors to $\calC$-sets)}

$\calF_\Obj$ maps a functor $F\colon \calC\to\cntSets$ to an \'{e}tale space $p_F\colon X_F\to\calC_\Obj$ equipped with an action $\alpha_F$ defined as follows.

Assume $F\colon \calC \to\cntSets$ is a computable functor. For $c\in\calC_\Obj$ we write $\equiv_c$ for the partial equivalence relation $F_{\Obj}(c)$, and for $n\in\IN$ we write $[n]_c=\{ m\in\IN\mid n \equiv_c m\}$ for the equivalence class of $n$ with respect to $\equiv_c$. Let $X_F\subseteq \calC_\Obj\times \calP(\IN)$ be the subspace defined as 
\[X_F = \{ \langle c, [n]_c \rangle \in \calC_\Obj \times \calP(\IN) \mid n \equiv_c n\}.\]
Define $p_F\colon X_F\to\calC_\Obj$ to be the projection on to the first coordinate. Define open sets
\begin{eqnarray*}
U_n &=& \{\langle c, [n]_c \rangle \in X_F\mid  n \equiv_c n \},\text{ and}\\
V_n &=& \{ c\in \calC_\Obj \mid n \equiv_c n \}.
\end{eqnarray*}
Define $s_n\colon V_n \to U_n$ as $s_n(c)=\langle c,[n]_c\rangle$. Then $s_n$ is the inverse for $p|_{U_n}$, which shows that $p$ is a computable local homeomorphism.

Define the action $\alpha_F:\subseteq \calC_\Mor \times X_F \to X_F$ as 
\[\alpha_F(f,\langle c, [n]_c \rangle) = \langle \tar(f), F_\Mor(f)(n)\rangle\]
where  for $F_\Mor(f) = \langle G_f, \equiv_\src, \equiv_\tar\rangle$ we set $F_\Mor(f)(n)=\{m\in\IN \mid G_f(n,m)\}$. Using the assumption that $F$ is a computable functor, it is straightforward to check that $\alpha_F$ is a computable action.

\subsection*{Morphisms (natural transformations to equivariant functions)}

$\calF_\Mor$ maps a natural transformation $\eta$ to an equivariant function $h$ as follows. 

Assume $F,F'\colon \calC \to\cntSets$ are computable functors, and that $\eta\colon \calC_\Obj \to \cntSets_\Mor$ corresponds to a computable natural transformation from $F$ to $F'$. Let $p\colon X \to \calC_\Obj$ and $q\colon Y \to \calC_\Obj$ be the $\calC$-sets $\calF_\Obj(F)$ and $\calF_\Obj(F')$, respectively (we omit the subscripts). Define $h\colon X\to Y$ as 
\[h(\langle c, [n]_c \rangle) = \langle c, \eta(c)(n) \rangle\]
where for $\eta(c) = \langle G, F_\Obj(c), F'_\Obj(c)\rangle$ we set $\eta(c)(n)=\{m\in\IN \mid G(n,m)\}$.

It is clear that $h$ is computable. We check that it is equivariant. Since $p$ and $q$ project the first coordinate and $h$ does not change it, it is clear that $p=q\circ h$. For $f\colon c\to d$ in $\calC_\Mor$, we have (with a slight abuse of notation in the second and third lines)
\begin{eqnarray*}
h(f\cdot \langle c, [n]_c \rangle) &=& h(\langle d, F_\Mor(f)(n)\rangle)\\
 &=& \langle d, \eta(d)(F_\Mor(f)(n))\rangle\\
  &=& \langle d, F'_\Mor(f)(\eta(c)(n))\rangle\\
   &=& f\cdot\langle c, \eta(c)(n)\rangle\\
  &=& f\cdot h(\langle c, [n]_c \rangle),
\end{eqnarray*}
hence $h$ is equivariant.

\subsection{From $\act(\calC)$ to $\cntSets^\calC$}\label{subsec:espace2functor}

The next two subsections define the functor  $\calG\colon \act(\calC) \to \cntSets^\calC$.

\subsection*{Objects ($\calC$-sets to functors)}

$\calG_\Obj$ maps an \'{e}tale space $p\colon X\to\calC_\Obj$ equipped with an action $\alpha$ to a functor $F\colon \calC\to\cntSets$ as follows.

Assume $p\colon X\to\calC_\Obj$ is a computable \'{e}tale space with sections $s_n\colon V_n \to U_n$ ($n\in\IN$), and computable action  $\alpha:\subseteq \calC_\Mor \times X \to X$. Define $F_\Obj\colon \calC_\Obj\to \cntSets_\Obj$ as
\[F_\Obj(c) = \{\langle n,m\rangle \in \IN\times\IN \mid c\in V_n\cap V_m \,\&\, s_n(c)=s_m(c)\}\]
and define $F_\Mor\colon\calC_\Mor\to \cntSets_\Mor$ as
\[F_\Mor(f)=\langle G_f, \equiv_\src,\equiv_\tar\rangle,\]
where
\begin{eqnarray*}
G_f &=& \{ \langle n,m\rangle \in\IN\times\IN \mid n \equiv_\src n \,\&\, m \equiv_\tar m \,\&\, \alpha(f,s_n(\src(f))) = s_m(\tar(f))\}\\
\equiv_\src &=& F_\Obj(\src(f))\\
\equiv_\tar &=& F_\Obj(\tar(f)).
\end{eqnarray*}

We show that $F\colon \calC\to\cntSets$ is a computable functor. It is clear that $F_\Obj(x)$ is a partial equivalence relation, hence $F_\Obj\colon \calC_\Obj\to \cntSets_\Obj$ is well-defined and easily seen to be computable by using Lemma~\ref{lem:sections}.

$F_\Mor\colon \calC_\Mor\to \cntSets_\Mor$ is also easily seen to be computable by using Lemma~\ref{lem:sections} and the fact that $\alpha$ being an action implies $p(\alpha(f,s_n(\src(f)))) = \tar(f)$, but we must check that $F_\Mor(f)$ satisfies the 5 conditions to be in $\cntSets_\Mor$. Condition 1 holds by definition. For condition 2, if $G_f(n,m)$ and $n\equiv_\src n'$, then $s_n(\src(f)) = s_{n'}(\src(f))$ by definition of $\equiv_\src$, hence $G_f(n',m)$ follows. Condition 3 is similar. For condition 4, if $G_f(n,m)$ and $G_f(n,m')$ then $\tar(f) \in V_m \cap V_{m'}$ and $s_m(\tar(f)) = s_{m'}(\tar(f))$ because $\alpha$ is single-valued, hence $m\equiv_\tar m'$. For condition 5, if $n\equiv_\src n$, then $\src(f) \in V_n$, hence $p(s_n(\src(f))) = \src(f)$, thus $\langle f, s_n(\src(f))\rangle \in dom(\alpha)$. Choose $m\in\IN$ with $\alpha( f, s_n(\src(f)))\in U_m$. Then $\tar(f)=p(\alpha( f, s_n(\src(f))))$ because $\alpha$ is an action. Thus $\tar(f)\in V_m$, which implies $m \equiv_\tar m$, and it follows from Lemma~\ref{lem:sections} that $\alpha(f,s_n(\src(f))) = s_m(\tar(f))$. Therefore, $G_f(n,m)$, and condition 5 is satisfied. 

Finally, it is straightforward to verify that $F_\Obj$ and $F_\Mor$ together satisfy the 4 conditions required to be a functor: the first two conditions hold by definition, and the last two follow from the properties of actions.

\subsection*{Morphisms (equivariant functions to natural transformations)}

$\calG_\Mor$ maps an equivariant function $h$ to a natural transformation $\eta_h$ as follows. 

Assume $h\colon X\to Y$ is a computable equivariant function from $p\colon X \to \calC_\Obj$ to $q\colon Y \to \calC_\Obj$. Let $s_n\colon V_n \to U_n$ ($n\in\IN$) be the sections of $p$, and $s'_n\colon V'_n \to U'_n$ the sections of $q$. Let $F, F'\colon \calC\to\cntSets$ be the computable functors corresponding to $p$ and $q$, respectively.

Define $\eta_h \colon\calC_\Obj \to \cntSets_\Mor$ as 
\[\eta_h(c) = \langle G,\equiv_\src,\equiv_\tar\rangle\]
where
\begin{itemize}
\item
$G=\{ \langle n,m\rangle \in\IN\times\IN \mid n \equiv_\src n \,\&\, m \equiv_\tar m \,\&\, h(s_n(c)) = s'_m(c)\}$
\item
$\equiv_\src = F_\Obj(c)$
\item
$\equiv_\tar = F'_\Obj(c)$
\end{itemize}
It is straightforward to see that $\eta_h$ is well-defined, and since $q(h(s_n(c)))=p(s_n(c)) = c$, Lemma~\ref{lem:sections} implies $h(s_n(c)) = s'_m(c)$ if and only if $h(s_n(c))\in U'_m$, hence $\eta_h$ is computable. 

We check that it is in fact a natural transformation $\eta_h \colon F \to F'$. For $f\colon c\to d$ in  $\calC_\Mor$,
\begin{eqnarray*}
&&\eta_h(d)\circ F_\Mor(f) \\
&=&\langle \{\langle n,m\rangle \mid (\exists k) [\cdots]\,\&\, f\cdot s_n(c) = s_k(d) \,\&\, h(s_k(d)) = s'_m(d) \}, F_\Obj(c), F'_\Obj(d)\rangle\\
 &=&\langle \{\langle n,m\rangle \mid [\cdots]\,\&\, h(f\cdot s_n(c)) = s'_m(d) \}, F_\Obj(c), F'_\Obj(d)\rangle\\
  &=&\langle \{\langle n,m\rangle \mid [\cdots]\,\&\, f\cdot h(s_n(c)) = s'_m(d) \}, F_\Obj(c), F'_\Obj(d)\rangle\\
&=& \langle \{\langle n,m\rangle \mid (\exists k) [\cdots]\,\&\, h(s_n(c)) = s'_k(c) \,\&\, f\cdot s'_k(c) = s'_m(d) \}, F_\Obj(c), F'_\Obj(d)\rangle\\
&=& F'_\Mor(f) \circ \eta_h(c)
\end{eqnarray*}
where $[\cdots]$ abbreviates $n\equiv n$, $m\equiv m$, $k\equiv k$ (when relevant, and with the appropriate subscripts). The other condition for being  a natural transformation is trivial.

\subsection{Natural isomorphism (from $\cntSets^\calC$ and back)}

We show that $\calG\circ\calF$ is actually equal to the identity functor on $\cntSets^\calC$.

Assume $F\colon \calC \to\cntSets$ is a computable functor, and let $\widehat{F} = \calG_\Obj(\calF_\Obj(F))$. Then $\widehat{F}_\Obj\colon \calC_\Obj\to \cntSets_\Obj$ is as follows:
\begin{eqnarray*}
\widehat{F}_\Obj(c) &=& \{\langle n,m\rangle \in \IN\times\IN \mid c\in V_n\cap V_m \,\&\, s_n(c)=s_m(c)\}\\
&=& \{\langle n,m\rangle \in \IN\times\IN \mid n \equiv_c n \,\&\, m \equiv_c m  \,\&\, \langle c,[n]_c\rangle=\langle c,[m]_c\rangle\}\\
&=& \{\langle n,m\rangle \in \IN\times\IN \mid n \equiv_c m \}\\
&=& F_{\Obj}(c),
\end{eqnarray*}
where $\equiv_c$ and $[\cdot]_c$ refer to $F_\Obj(c),$ and $s_n\colon V_n \to U_n$ are the sections of $\calF_\Obj(F)$.

For morphisms, assuming $F_\Mor(f) = \langle G_f, \equiv_\src, \equiv_\tar\rangle$, we have that $\widehat{F}_\Mor(f) =\langle \widehat{G}_f, \equiv_\src,\equiv_\tar\rangle$, where $\equiv_\src = F_\Obj(\src(f)) = \widehat{F}_\Obj(\src(f))$, and $\equiv_\tar = F_\Obj(\tar(f)) =  \widehat{F}_\Obj(\tar(f))$. Furthermore, assuming $n \equiv_\src n$ and $m \equiv_\tar m$, we have
\begin{eqnarray*}
\widehat{G}_f(n,m) &\iff& \alpha_F(f,s_n(\src(f))) = s_m(\tar(f))\\
&\iff& \alpha_F(f,\langle \src(f),[n]_{\src(f)}\rangle) = \langle \tar(f),[m]_{\tar(f)}\rangle \\
&\iff& \langle \tar(f), F_\Mor(f)(n)\rangle= \langle \tar(f),[m]_{\tar(f)}\rangle \\
&\iff& F_\Mor(f)(n)= [m]_{\tar(f)} \\
&\iff& G_f(n,m),
\end{eqnarray*}
where $\alpha_F$ is the action of $\calF_\Obj(F)$. Therefore, $\widehat{F} = F$.

Next, assume $F,F'\colon \calC \to\cntSets$ are functors and that $\eta\colon \calC_\Obj \to \cntSets_\Mor$ corresponds to a natural transformation from $F$ to $F'$. Let $\widehat{\eta} = \calG_\Mor(\calF_\Mor(\eta))$. Fix $c\in \calC_\Obj$ and assume $\eta(c) = \langle G, F_\Obj(c), F'_\Obj(c)\rangle$ and $h = \calF_\Mor(\eta)$. Then $\widehat{\eta}(c) = \langle \widehat{G}, F_\Obj(c), F'_\Obj(c)\rangle$, where for $\langle n,n\rangle \in F_\Obj(c)$ and $\langle m,m\rangle \in F'_\Obj(c)$, we have
\begin{eqnarray*}
\widehat{G}(n,m) &\iff& h(s_n(c)) = s'_m(c)\\
&\iff& h(\langle c,[n]_c\rangle) = \langle c,[m]'_c\rangle\\
&\iff& \langle c, \eta(c)(n)\rangle = \langle c,[m]'_c\rangle\\
&\iff& \eta(c)(n) = [m]'_c\\
&\iff& \{ b\in\IN \mid G(n,b)\} = \{ b\in\IN \mid b \equiv'_c m\}\\
&\iff& G(n,m),
\end{eqnarray*}
where $s_n$ refers to $\calF_\Obj(F)$, and $\equiv'_c$, $[\cdot]'_c$, and $s'_m$ refer to $\calF_\Obj(F')$. Thus, $\widehat{\eta}=\eta$.

\subsection{Natural isomorphism (from $\act(\calC)$ and back)}

We show that $\calF\circ\calG$ is naturally isomorphic to the identity functor on $\act(\calC)$.

Assume $p\colon X\to\calC_\Obj$ is a computable \'{e}tale space with sections $s_n\colon V_n \to U_n$ ($n\in\IN$), and computable action  $\alpha:\subseteq \calC_\Mor \times X \to X$. Let $F\colon \calC\to\cntSets$ be the functor obtained from applying $\calG_\Obj$ to $X$ (see Section~\ref{subsec:espace2functor}).

We calculate the \'{e}tale space $p_F\colon X_F\to\calC_\Obj$ with action $\alpha_F$ corresponding to $\calF_\Obj(F)$ (see Section~\ref{subsec:functor2espace}). We have 
\begin{eqnarray*}
X_F &=& \{ \langle c, [n]_c \rangle \in \calC_\Obj \times \calP(\IN) \mid n \equiv_c n \}\\
&=& \{ \langle c, [n]_c \rangle \in \calC_\Obj \times \calP(\IN) \mid c\in V_n\},
\end{eqnarray*}
where for $c\in V_n$,
\begin{eqnarray*}
[n]_c &=&\{ m\in\IN\mid n \equiv_c m\}\\
&=&\{ m\in\IN\mid  c\in V_m \,\&\, s_n(c)=s_m(c)\}\\
&=&\{ m\in\IN\mid  s_n(c) \in U_m\}
\end{eqnarray*}
and $p_F\colon X_F\to\calC_\Obj$ is the projection on to the first coordinate.

After unwinding the definitions, the action $\alpha_F:\subseteq \calC_\Mor \times X_F \to X_F$ is
\begin{eqnarray*}
\alpha_F(f,\langle c, [n]_c \rangle) &=& \langle \tar(f), \{m'\in\IN \mid \tar(f)\in V_{m'} \,\&\, \alpha(f,s_n(c)) = s_{m'}(\tar(f))\} \rangle\\
&=&\langle \tar(f),\{m'\in\IN \mid \alpha(f,s_n(c)) \in U_{m'}\}\rangle\\
&=&\langle \tar(f),\{m'\in\IN \mid s_{m}(p(\alpha(f,s_n(c)))) \in U_{m'}\}\rangle\\
&=&\langle \tar(f),\{m'\in\IN \mid s_m(\tar(f)) \in U_{m'}\}\rangle\\
&=&\langle \tar(f),[m]_{\tar(f)}\}\rangle
\end{eqnarray*}
where by assumption $\src(f) = c$ and $c\in V_n$, and $m\in\IN$ is any number satisfying $\alpha(f,s_n(c)) \in U_m$.

Define $\theta_X\colon X \to X_F$ as $\theta_X(x) = \langle p(x), [n]_{p(x)} \rangle$, where $n\in \IN$ satisfies $x\in U_n$ (if $x \in U_n \cap U_m$ then $[n]_{p(x)} = [m]_{p(x)}$ so the definition of $\theta_X(x)$ does not depend on the choice of $n$). It is clear that $\theta_X$ is computable, and it is equivariant because
\begin{eqnarray*}
\theta_X(\alpha(f,x)) &=& \langle p(\alpha(f,x)), [m]_{p(\alpha(f,x))} \rangle\quad \text{ (where $m\in\IN$ satisfies $\alpha(f,x) \in U_m$)}\\
  &=& \langle \tau(f), [m]_{\tau(f)} \rangle\\
  &=& \alpha_F(f,\langle p(x), [n]_{p(x)} \rangle)\\
  && \quad \text{ (as shown above, because $\src(f)=p(x)$ and $p(x)\in V_n$)}\\
  &=& \alpha_F(f,\theta_X(x)).
\end{eqnarray*}

Next, define $\theta_X'\colon X_F \to X$ as $\theta_X'(\langle c,[n]_c\rangle) = s_n(c)$. Note that $\theta_X'$ is well-defined because $s_n(c)=s_m(c)$ for each $m\in [n]_c$. It is clear that $\theta_X'$ is computable, and it is equivariant because:
\begin{eqnarray*}
\theta_X'(\alpha_F(f,\langle c, [n]_c \rangle)) &=& \theta_X'(\langle \tar(f),[m]_{\tar(f)}\rangle)\\
&& \quad \text{ (where $m\in\IN$ satisfies $\alpha(f,s_n(c)) \in U_m$)}\\
&=& s_m(\tar(f))\\
&=& \alpha(f,s_n(c)) \quad\text{ (by Lemma~\ref{lem:sections} because $\alpha(f,s_m(c)) \in U_m$)}\\
&=& \alpha(f,\theta_X'(\langle c, [n]_c \rangle)).
\end{eqnarray*}
Furthermore, $\theta_X$ and $\theta_X'$ are inverses because:
\begin{eqnarray*}
\theta_X'(\theta_X(x)) &=& \theta_X'(\langle p(x), [n]_{p(x)} \rangle)\quad\text{ (where $n\in\IN$ satisfies $x \in U_n$)}\\
 &=& s_n(p(x))\\
 &=& x,
\end{eqnarray*}
and
\begin{eqnarray*}
\theta_X(\theta_X'(\langle c,[n]_c\rangle)) &=& \theta_X(s_n(c))\\
&=& \langle p(s_n(c)), [n]_{p(s_n(c))} \rangle\quad\text{ (because $s_n(c) \in U_n$)}\\
&=& \langle c, [n]_{c} \rangle.
\end{eqnarray*}
Therefore, $X$ and $X_F$ are computably isomorphic $\calC$-sets.

Finally, we prove that $\theta\colon 1_{\act(\calC)}\to \calF\circ\calG$ is a natural transformation (it automatically follows that the inverse $\theta'$ is also natural). Assume $h\colon X\to Y$ is a continuous equivariant function from $p\colon X\to\calC_\Obj$ to $q\colon Y\to\calC_\Obj$. Let $s_n\colon V_n \to U_n$ ($n\in\IN$) be the sections of $p$, and $s'_n\colon V'_n \to U'_n$ the sections of $q$. We write $\eta_h \colon F \to F'$ for $\calG_\Mor(h)$ and $h' \colon X_F \to Y_{F'}$ for $\calF_\Mor(\eta_h)$. Then for $c\in\calC_\Obj$ with $c\in V_n$ we have
\[h'(\langle c, [n]_c \rangle) = \langle c, \eta_h(c)(n) \rangle\]
where 
\begin{eqnarray*}
m \in \eta_h(c)(n) &\iff& \langle n,n\rangle \in  F_\Obj(c) \,\&\, \langle m,m\rangle \in  F'_\Obj(c) \,\&\, h(s_n(c)) = s'_m(c)\\
&\iff& c \in V_n \cap V'_m \,\&\, h(s_n(c)) = s'_m(c)\\
&& \qquad \text{ (by definition of $F_\Obj$ and $F'_\Obj$)}\\
&\iff& c \in V_n\cap V'_m \,\&\, h(s_n(c)) \in U'_m\\
&& \qquad \text{ (by Lemma~\ref{lem:sections} because $q(h(s_n(c))) = c$)}.
\end{eqnarray*}

Assume $x\in X$, and fix $n \in \IN$ with $x\in U_n$. Since $s_n(p(x))=x$, if $m\in\eta_h(p(x))(n)$ then $h(x)=h(s_n(p(x))) \in U'_m$. Conversely, if $h(x) \in U'_m$, then since $p(x) \in V_n$ and $p(x)=q(h(x)) \in V'_m$ and $h(s_n(p(x))) = h(x) \in U'_m$, it follows that $m \in \eta_h(p(x))(n)$. Therefore, $m\in \eta_h(p(x))(n)$ if and only if $h(x)\in U'_m$. 

Now fix any $m\in\IN$ with $h(x)\in U'_m$. Since $s'_m(p(x)) = s'_m(q(h(x))) = h(x)$, we have $m' \in [m]'_{p(x)}$ if and only if $h(x)=s'_m(p(x)) \in U'_{m'}$.

It follows that $\eta_h(p(x))(n) =  [m]'_{p(x)}$, hence
\begin{eqnarray*}
h'(\theta_X(x)) &=& h'(\langle p(x), [n]_{p(x)} \rangle)\\
&=& \langle p(x), \eta_h(p(x))(n) \rangle\\
&=& \langle p(x), [m]'_{p(x)} \rangle\\
&=& \langle q(h(x)), [m]'_{q(h(x))} \rangle\\
&=& \theta_Y(h(x)). 
\end{eqnarray*}
Therefore, $\theta\colon 1_{\act(\calC)} \to \calF\circ\calG$ is a natural isomorphism.


\bibliographystyle{amsplain}
\bibliography{myrefs}

\end{document}